\documentclass{amsart}
\usepackage{amscd,amssymb}

\newtheorem{theorem}{Theorem}[section]
\newtheorem{proposition}[theorem]{Proposition}
\newtheorem{lemma}[theorem]{Lemma}
\newtheorem{corollary}[theorem]{Corollary}
\theoremstyle{definition}
\newtheorem{definition}[theorem]{Definition}

\theoremstyle{remark}

\numberwithin{equation}{section}

\DeclareMathOperator{\mor}{mor}
\DeclareMathOperator{\Hom}{Hom}

\DeclareMathOperator{\sho}{ho}

\DeclareMathOperator{\colim}{colim}

\DeclareMathOperator{\coker}{coker}

\DeclareMathOperator{\Arr}{Arr}

\newcommand{\boxprod}{\mathbin\square}

\newcommand{\cat}[1]{\mathcal{#1}}

\newcommand{\Z}{\mathbb{Z}}

\newcommand{\Mod}{\text{-mod}}

\newcommand{\ev}{\textup{Ev}}

\newcommand{\mathcolon}{\colon\,}

\newcommand{\uc}{\textup{:}}
\newcommand{\ulp}{\textup{(}}
\newcommand{\urp}{\textup{)}}
\newcommand{\usc}{\textup{;}}

\begin{document}
\title{Smith ideals of structured ring spectra} 

\date{\today}

\author{Mark Hovey}
\address{Department of Mathematics \\ Wesleyan University
\\ Middletown, CT 06459}
\email{hovey@member.ams.org}


\begin{abstract}
Pursuing ideas of Jeff Smith, we develop a homotopy theory of ideals
of monoids in a symmetric monoidal model category.  This includes
Smith ideals of structured ring spectra and of differential graded
algebras.  Such Smith ideals are NOT subobjects, and as a result the
theory seems to require us to consider all Smith ideals of all monoids
simultaneously, rather then restricting to the Smith ideals of one
particular monoid. However, we can take a quotient by a Smith ideal
and get a monoid homomorphism.  In the stable case, we show that this
construction is part of a Quillen equivalence between a model category
of Smith ideals and a model category of monoid homomorphisms.  
\end{abstract}

\maketitle

\section*{Introduction}

There has been a great deal of interest in the last several decades in
the theory of structured ring spectra in algebraic topology.  In
simple terms, these are cohomology theories with cup products that are
infinitely homotopy associative, though we now think of them just as
monoids in a suitable model category of spectra.  In trying to develop
the ring theory of these ring spectra, one runs into a basic problem
right away: a ring spectrum $R$ has no elements.  Thus even a simple
concept like an ideal of a ring spectrum has not been developed.  Of
course, a ring spectrum $R$ has associated to it its homotopy ring
$\pi_{*}R$, and we can talk of ideals here.  But right away we run
into trouble.  For $R=S$, the sphere spectrum, and the element $2\in
\pi_{0}S$, the cofiber of the times $2$ map, the mod $2$ Moore
spectrum, is not a ring spectrum even up to homotopy.  So $(2)$ cannot
be an ideal of $S$, though it is an ideal of $\pi_{*}S$.  

In a talk given on June 6, 2006, Jeff Smith outlined a theory of
ideals of ring spectra.  He did much more than this, using his theory
to discuss algebraic K-theory of pushouts.  Bob Bruner kindly sent me
his notes from the talk, and some of his thoughts with Dan
Isaksen~\cite{bruner-isaksen} about Smith's work.  This paper is my
interpretation of some of that work.  I stress that the material in
this paper was essentially the background material in Smith's talk, so
there is much more left to be done.

Obviously we cannot think of an ideal in a ring spectrum as a
collection of elements.  We should instead think of it as a map
$f\mathcolon I\xrightarrow{}R$.  If $I$ is a two-sided ideal, this
should be a monomorphism of $R$-bimodules.  However, we know that
``monomorphism'' is not a good homotopy-theoretic concept, because
every map should be homotopic to a monomorphism.  So a Smith ideal
should be a map $f\mathcolon I\xrightarrow{}R$ of $R$-bimodules that
is trying to be a monomorphism in some way. 

We could continue in this vein, but let us instead think for a moment
about the category $\Arr \cat{C}$ of maps in a symmetric monoidal
category $\cat{C}$.  A Smith ideal is going to be a certain kind of
object in $\Arr \cat{C}$, and since it is related to monoids in
$\cat{C}$, maybe a Smith ideal should simply be a monoid in $\Arr
\cat{C}$.  For this, we need a symmetric monoidal structure on $\Arr
\cat{C}$ induced by the one from $\cat{C}$.  There is one obvious one,
the \textbf{tensor product monoidal structure}, in which the monoidal
product of $f\mathcolon X_{0}\xrightarrow{}X_{1}$ and $g\mathcolon
Y_{0}\xrightarrow{}Y_{1}$ is 
\[
f\otimes g\mathcolon X_{0}\otimes Y_{0} \xrightarrow{} X_{1} \otimes Y_{1}.
\]
However, a monoid for this monoidal structure is just a monoid
homomorphism, so this cannot be the right monoidal structure.  

The correct monoidal structure on $\Arr \cat{C}$ is the
\textbf{pushout product monoidal structure}, in which the monoidal
product $f\boxprod g$ of $f$ and $g$ is the map 
\[
f\boxprod g\mathcolon (X_{0}\otimes Y_{1}) \amalg_{(X_{0}\otimes
Y_{0})} (X_{1}\otimes Y_{0}) \xrightarrow{} X_{1} \otimes Y_{1}.
\]
This is the same pushout product used in the definition of a symmetric
monoidal model structure.  A monoid for this monoidal structure turns
out to be a monoid $R$, an $R$-bimodule $I$, and a map of
$R$-bimodules $j\mathcolon I\xrightarrow{}R$, such that 
\[
\mu (1\otimes j)=\mu (j\otimes 1)\mathcolon I\otimes I\xrightarrow{}I,
\]
where $\mu$ denotes both the right and left multiplication of $R$ on
$I$.   This is a Smith ideal and is equivalent (by using work of
Bruner and Isaksen~\cite{bruner-isaksen}) to the definition Smith
gave in his talk in terms of enriched functors.  

In this paper, we first construct the pushout product and tensor
product monoidal structures on $\Arr \cat{C}$ in
Section~\ref{sec-arrow}, and show the perhaps suprising result that
the cokernel is a symmetric monoidal functor from the pushout product
monoidal structure to the tensor product monoidal structure.  This
means that the cokernel of a Smith ideal is a ring spectrum.  

We then need to consider homotopy theory.  In
Section~\ref{sec-model-projective} and
Section~\ref{sec-model-injective} we develop two model category
structures on $\Arr \cat{C}$.  In the projective model structure, a
morphism in $\Arr \cat{C}$ is a fibration or weak equivalence if and
only if its two components are so in $\cat{C}$.  The projective model
structure is compatible with the tensor product monoidal structure.
In the injective model structure, a morphism is a cofibration or weak
equivalence if and only if its two components are so in $\cat{C}$.
The injective model structure is compatible with the pushout product
monoidal structure.  As a result, we get model categories of Smith
ideals and monoid homomorphisms, and the cokernel is a left Quillen
functor from Smith ideals to monoid homomorphisms.  To take the
homotopically meaningful quotient of a ring spectrum $R$ by a Smith
ideal $j\mathcolon I\xrightarrow{}R$, we must first take a cofibrant
approximation $j'\mathcolon I'\xrightarrow{}R'$ in the model category
of Smith ideals, and then take $R'/I'$.  Note that $R'$ is weakly
equivalent as a monoid to $R$, but of course it is not $R$.  

In Section~\ref{sec-stable} we prove that the cokernel is a Quillen
\textbf{equivalence} in case our model category $\cat{C}$ is stable,
as for example any model category of spectra.  Thus, in the stable
case, a Smith ideal is the same information homotopically as a monoid
homomorphism.  This is a bit disconcerting, since this says, for
example, that every structured ring spectrum $R$ is weakly equivalent
to the quotient of the sphere $S$ by some Smith ideal.  This would be
like saying that every ring is a quotient of the integers $\Z $.  But
we have to remember that we are free to add an enormous contractible
ring spectrum to $S$ and take a Smith ideal of that, and this is why
we can get $R$ as a quotient of $S$ up to homotopy.  There is also an
appendix where we discuss some technical issues.

There is much work still to be done.  For example, given any map
$f\mathcolon I\xrightarrow{}R$, what is the Smith ideal generated by
$f$?  The answer must be the free monoid $Tf$ on $f$ in the pushout
product monoidal structure.  The problem is that this is a Smith ideal
of the free monoid $TR$ on $R$, not of $R$ itself.  The author has
tried various ways to convert this to a Smith ideal of $R$, without
success.  It is possible that we have to accept that the ideal
generated by $2\mathcolon S\xrightarrow{}S$ is an ideal of $TS=S[x]$,
not of $S$.  Even in that case, it would be extremely interesting to
know the quotient of $S[x]$ by $T (2)$.

A related question is to fix the quotient issue.  We have the wrong
definition of quotient if every ring spectrum is a quotient of $S$.
Our weakening of the definition of ideal means that we should
strengthen the definition of quotient.   One way to do it is to define
a monoid homomophism $p\mathcolon R\xrightarrow{}S$ to be a
\textbf{strong quotient} if the map
\[
S\otimes_{R} QN \xrightarrow{}N
\]
is a weak equivalence for all fibrant $S$-modules $N$, where $Q$
denotes cofibrant replacement in the category of $R$-modules.  This
would mean that the homotopy category of $S$-modules is a fully
faithful subcategory of the homotopy category of $R$-modules.  Is this
a useful notion?  Is it possible to use Smith ideals to help classify
strong quotients of ring spectra?  We don't know the answers.  

Finally, we have not dealt with the commutative situation at all.  It
is generally much more difficult to determine whether commutative
monoids in a symmetric monoidal model category inherit a model
structure, as there are several well-known instances where it is
false.  But if they do, one might hope to define a commutative Smith
ideal as a commutative monoid in the arrow category.  This would give
us a potentially different notion of ideal, and we would like to know
how different.  

Obiovusly, this talk owes everything to Jeff Smith's 2006 lecture, and
the author thanks him for having such wonderful ideas.  The author is
certain that he has not come close to the depth of Smith's vision
about the subject.  The author also thanks Bob Bruner and Dan Isaksen
for sharing their very helpful thoughts on Smith ideals.  

Throughout this paper, $\cat{C}$ will denote a bicomplete closed
symmetric monoidal category with monoidal product $A\otimes B$, closed
structure $\Hom (A,B)$, and unit $S$.  Usually $\cat{C}$ will also be
a model category, and in that case we will always assume the model
structure is compatible the monoidal structure, so that $\cat{C}$ is
a symmetric monoidal model category.  Facts about model categories can
generally be found in~\cite{hovey-model}, among other places, but we
try to cite specific facts more precisely.  

\section{The arrow category of a symmetric monoidal model
category}\label{sec-arrow}

In this section, we show that there are two different closed symmetric
monoidal structures on the arrow category $\Arr \cat{C}$ of our closed
symmetric monoidal category $\cat{C}$.  We also discuss the cokernel
functor, which is a symmetric monoidal functor from one of these
structures to the other. Finally, we discuss monoids and modules in
each of these symmetric monoidal structures.  We postpone all
discussion of homotopy theory to the next section.  

Recall that an object $\Arr \cat{C}$ is simply a map $f$ of $\cat{C}$.
For notational convenience, we will often write $f$ as $f\mathcolon
X_{0}\xrightarrow{}X_{1}$. A map $\alpha \mathcolon
f\xrightarrow{}g$ in $\Arr \cat{C}$ is then a commutative square
\[
\begin{CD}
X_{0} @>f>> Y_{1} \\
@V\alpha_{0}VV @VV\alpha_{1}V \\
Y_{0} @>>g> Y_{1}.
\end{CD}
\]
Note that $\Arr \cat{C}$ is the category of functors from the category
$\cat{J}$ with two objects $0$ and $1$ and one non-identity map
$0\xrightarrow{}1$ to $\cat{C}$.  Hence, since $\cat{C}$ is
bicomplete, so is $\Arr \cat{C}$, with limits and colimits taken
objectwise. In addition, if $\cat{C}$ is locally presentable, so is
$\Arr \cat{C}$, as a category of small diagrams in a locally
presentable category~\cite{adamek-rosicky}.

There are obvious functors $\ev_{0},\ev_{1}\mathcolon \Arr
\cat{C}\xrightarrow{}\cat{C}$, and these have left and right
adjoints.  We leave the following lemma to the reader.  

\begin{lemma}\label{lem-eval}
Suppose $\cat{C}$ is a closed symmetric monoidal category.  The
evaluation functors $\ev_{0},\ev_{1}\mathcolon \Arr
\cat{C}\xrightarrow{}\cat{C}$ have left adjoints $L_{0},L_{1}$ and
right adjoints $U_{0},U_{1}$.  We have
\[
L_{0} (X)=U_{1}X= 1_{X}, L_{1} (X)=0\xrightarrow{}X, \text{ and }
U_{0} (X)=X\xrightarrow{}*.  
\]
\end{lemma}

We now discuss monoidal structures on $\Arr \cat{C}$, of which there
are at least two. 

\begin{theorem}\label{thm-arrow-monoidal}
Let $\cat{C}$ be a closed symmetric monoical category.  The category
$\Arr \cat{C}$ has two different closed symmetric monoidal structures.
In the \textbf{tensor product monoidal structure}, the monoidal
product of $f\mathcolon X_{0}\xrightarrow{}X_{1}$ and $g\mathcolon
Y_{0}\xrightarrow{}Y_{1}$ is given by
\[
f\otimes g \mathcolon X_{0} \otimes Y_{0} \xrightarrow{} X_{1} \otimes
Y_{1}.  
\]
The unit is $L_{0}S$, and the closed structure is given by the
projection map
\[
\Hom_{\otimes} (f,g) = \Hom (X_{0},Y_{0}) \times_{\Hom (X_{0},Y_{1})}
\Hom (X_{1},Y_{1}) \xrightarrow{} \Hom (X_{1}, Y_{1}).
\]
In the \textbf{pushout product monoidal structure}, the monoidal
product of $f$ and $g$ is the pushout product
\[
(X_{0}\otimes Y_{1}) \amalg_{(X_{0}\otimes Y_{0})} (X_{1}\otimes
Y_{0}) \xrightarrow{} X_{1} \otimes Y_{1}.
\]
The unit is $L_{1}S$, and the closed structure $\Hom_{\boxprod} (f,g)$
is given by  
\[
\Hom_{\boxprod} (f,g) = \Hom (X_{1},Y_{0}) \xrightarrow{} \Hom
(X_{0},Y_{0}) \times_{\Hom (X_{0},Y_{1})} \Hom (X_{1},Y_{1}).
\]
\end{theorem}

We note that $\Arr \cat{C}$ itself, with either closed symmetric
monoidal structure above, is an appropriate input for
Theorem~\ref{thm-arrow-monoidal}.  That is, we can iterate the
construction and get various closed symmetric monoidal structures on
$\Arr \Arr \cat{C}$, the category of commutative squares in
$\cat{C}$.  We can of course continue this iteration.  

\begin{proof}
We leave the majority of this proof to the reader.  The most annoying
thing to construct is the associativity isomorphism for the pushout
product monoidal structure.  Here the idea is
to use the properties of colimits to conclude that both $(f\boxprod
g)\boxprod h$ and $f\boxprod (g\boxprod h)$ are isomorphic to 
\[
\colim_{(i,j,k)\neq (1,1,1)} (X_{i}\otimes Y_{j}\otimes
Z_{k} \xrightarrow{} X_{1} \otimes Y_{1} \otimes
Z_{1},
\]
where $h\mathcolon Z_{0}\xrightarrow{}Z_{1}$.
\end{proof}

It is useful to record how $L_{0}$ and $L_{1}$ interact with these
symmetric monoidal structures.  

\begin{lemma}\label{lem-eval-tensor}
Let $\cat{C}$ be a closed symmetric monoidal category.  As a functor
to the tensor product monoidal structure, $L_{0}\mathcolon
\cat{C}\xrightarrow{}\Arr \cat{C}$ is symmetric monoidal, whereas
$L_{1}X \otimes f=L_{1} (X\otimes \ev_{1}f)$.  As a functor to the
pushout product monoidal structure, $L_{1}$ is symmetric monoidal,
whereas $L_{0} (X)\boxprod f=L_{0} (X\otimes \ev_{1}f)$.
\end{lemma}

Again, we leave the proof to the reader.  

Since we are interested in exploring the relationship between ideals
and quotients, we need to understand the functor $\coker \mathcolon
\Arr \cat{C}\xrightarrow{}\Arr \cat{C}$.  Here we need to assume
$\cat{C}$ is pointed.  The cokernel functor is defined by 
\[
\coker (f\mathcolon A\xrightarrow{}B) = (B\xrightarrow{}\coker f),
\]
where we rely on context to distinguish between $\coker f$ as a map in
$\cat{C}$ and as an object in $\cat{C}$.

We can now see the importance of the two different symmetric monoidal
structures on $\Arr \cat{C}$. 

\begin{theorem}\label{thm-cokernel-monoidal}
Suppose $\cat{C}$ is a pointed closed symmetric monoidal category.
The functor $\coker \mathcolon \Arr \cat{C}\xrightarrow{}\Arr \cat{C}$
is a \ulp strongly\urp symmetric monoidal functor from the pushout
product monoidal structure to the tensor product monoidal structure.
Its right adjoint is the kernel.
\end{theorem}

\begin{proof}
First note that the cokernel preserves the units, since 
\[
\coker (0\xrightarrow{}S) = (S\xrightarrow{=}S).  
\]
To see the cokernel is monoidal, we use the fact that pushouts commute
with each other.  More precisely, we have the following commutative
diagram, for maps $f\mathcolon X_{0}\xrightarrow{}X_{1}$ and
$g\mathcolon Y_{0}\xrightarrow{}Y_{1}$.  
\[
\begin{CD}
X_{1}\otimes Y_{1} @<<< X_{0} \otimes Y_{1} @>>> 0 \\
@AAA @AAA @| \\
X_{1} \otimes Y_{0} @<<< X_{0} \otimes  Y_{0} @>>> 0 \\
@| @VVV @| \\
X_{1} \otimes Y_{0} @= X_{1} \otimes Y_{0} @>>> 0
\end{CD}
\]
If we take vertical pushouts in this diagram, we get the diagram 
\[
X_{1} \otimes Y_{1} \xleftarrow{f\boxprod g} (X_{0}\otimes
Y_{1})\amalg_{X_{0}\otimes Y_{0}} (X_{1} \otimes Y_{0}) \xrightarrow{}
0,
\]
whose pushout is $\coker (f\boxprod g)$.  On the other hand, if we
take the horizontal pushouts instead, we get the diagram 
\[
\coker f \otimes Y_{1} \xleftarrow{} \coker f \otimes Y_{0}
\xrightarrow{} 0,
\]
whose pushout is $\coker f\otimes \coker g$.  Since pushouts commute
with each other, we see that 
\[
\coker (f\boxprod g) \cong \coker f\otimes \coker g,
\]
both as objects in $\cat{C}$ and as maps in $\cat{C}$.  We leave the
check that the required coherence diagrams commute to the reader, as
also the proof that the kernel is right adjoint to the cokernel.  
\end{proof}

We now consider the monoids and modules in our two symmetric monoidal
structures on $\Arr \cat{C}$.  

\begin{proposition}\label{prop-tensor-monoid}
Let $\cat{C}$ be a closed symmetric monoidal category.  A monoid in
the tensor product monoidal structure on $\Arr \cat{C}$ is simply a
monoid homomorphism $p\mathcolon R_{0}\xrightarrow{}R_{1}$ in
$\cat{C}$.  A module over the monoid $p$ is an $R_{0}$-module $M_{0}$,
an $R_{1}$-module $M_{1}$, and an $R_{0}$-module map $f\mathcolon
M_{0}\xrightarrow{}M_{1}$, where $M_{1}$ is an $R_{0}$-module by
restricting scalars through $p$.
\end{proposition}

Said another way, with the tensor product monoidal structure, the
monoids in the arrow category of $\cat{C}$ are the arrows in the
monoid category of $\cat{C}$.  

\begin{proof}
A monoidal structure on $p$ consists of a unit map 
\[
\begin{CD}
S @= S \\
@VVV @VVV \\
R_{0} @>>p> R_{1}
\end{CD}
\]
and a multiplication map
\[
\begin{CD}
R_{0}\otimes R_{0} @>p\otimes p>> R_{1}\otimes R_{1} \\
@VVV @VVV \\
R_{0} @>>p> R_{1}.
\end{CD}
\]
These are just equivalent to a unit and a multiplication for $R_{0}$ and
$R_{1}$ that are preserved by $p$.  The associativity and unit diagrams
says that the multiplications on $R_{0}$ and $R_{1}$ are associative and
unital.  The proof for modules is similar. 
\end{proof}

The monoids and modules in the pushout product monoidal structure are
much more interesting. 

\begin{definition}\label{defn-Smith}
Let $\cat{C}$ be a closed symmetric monoidal category.  A
\textbf{Smith ideal} in $\cat{C}$ is a monoid $j\mathcolon
I\xrightarrow{}R$ in the pushout product monoidal structure on $\Arr
\cat{C}$, which we often denote $(R,I)$.  Given a Smith ideal $(R,I)$,
a (right) \textbf{$(R,I)$-module} is a right module in the pushout
product monoidal structure on $\Arr \cat{C}$ over the monoid $j$.
\end{definition}
 
The ``Smith'' in Smith ideal is Jeff Smith.  We will discuss the
relationship between our definition and the definition given by Smith
below.  

Our first job is to unwind these definitions. 

\begin{proposition}\label{prop-ideal}
Let $\cat{C}$ be a closed symmetric monoidal category.  A Smith ideal
$j\mathcolon I\xrightarrow{}R$ in $\cat{C}$ is equivalent to a monoid
$R$ in $\cat{C}$, an $R$-bimodule $I$ in $\cat{C}$, and a morphism
$j\mathcolon I\xrightarrow{}R$ of $R$-bimodules such that the diagram
below commutes.
\[
\begin{CD}
I\otimes_{R}I @>j\otimes 1>> R\otimes_{R} I \\
@Vi\otimes j VV @VV\cong V \\
I\otimes_{R} R @>>\cong > I
\end{CD}
\]
\end{proposition}

Of course the tensor products in the commutative diagram above are
tensor products of bimodules, so involve both the right and left
action of $R$ on $I$.  

\begin{proof}
A monoid structure on $j$ is given by a multiplication map $\mu
\mathcolon j\boxprod j\xrightarrow{}j$ and a unit $\eta \mathcolon
L_{1}S\xrightarrow{}j$ making the usual associativity and unit
diagrams commute.  Writing $j\mathcolon I\xrightarrow{}R$, we see that
$\eta$ is equivalent to a map $\eta \mathcolon S\xrightarrow{}R$, and
$\mu$ is equivalent to the commutative diagram below.
\[
\begin{CD}
(I\otimes R) \amalg_{I\otimes I} (R\otimes I) @>>> R\otimes R \\
@VVV @VVV \\
I @>>j> R
\end{CD}
\]
Thus the existence of $\mu$ is equivalent to a multiplication on $R$,
a left and right multiplication of $R$ on $I$ that agree on $I\otimes
I$, and the fact that $j$ preserves those multiplications. If we write
out $j\boxprod (j\boxprod j)$ carefully, we see that its domain is an
iterated pushout involving $I\otimes (R\otimes R)$, $R\otimes
(I\otimes R)$, and $R\otimes (R\otimes I)$, and its codomain is
$R\otimes (R\otimes R)$.  Similarly, the domain of $(j\boxprod
j)\boxprod j$ is an iterated pushout involving $(I\otimes R)\otimes
R$, $(R\otimes I)\otimes R$, and $(R\otimes R)\otimes I$, and the
codomain is $(R\otimes R)\otimes R$.  Therefore, the associativity
diagram for $\mu$ is equivalent to associativity of the left
multiplication of $R$ on $I$, the fact that the left and right
multiplications of $R$ and $I$ commute with each other, associativity
of the right multiplication of $R$ on $I$, and the associativity of
the multiplication on $R$.  The unit diagrams are equivalent to $\eta$
acting as a left and right unit on $R$ and on $I$.
\end{proof}

It was pointed out by Bruner and Isaksen~\cite{bruner-isaksen} that
the data above, an $R$-bimodule map $j\mathcolon I\xrightarrow{}R$
making the diagram in Proposition~\ref{prop-ideal} commute, are
equivalent to a monoid $R$ and a $\cat{C}$-category $\cat{I}$ with two
objects $a$ and $b$ where $\cat{I} (a,a)=\cat{I} (b,b)=\cat{I}
(a,b)=R$ with all compositions involving only these morphism objects
being multiplication in $R$.  This was the definition of an ideal of
$R$ given by Jeff Smith in his talk of June 6, 2006.  The $R$-bimodule
$I$ is then $\cat{I} (b,a)$.

A map of Smith ideals $\alpha \mathcolon (R,I)\xrightarrow{} (R',I')$
is of course just a map of monoids in $\Arr \cat{C}$.  This unwinds to
a map of monoids $\alpha_{1}\mathcolon R\xrightarrow{}R'$ and a map of
$R$-bimodules $\alpha_{0}\mathcolon I\xrightarrow{}I'$, where $I'$ is
an $R$-bimodule through restriction of scalars, such that
$\alpha_{1}j=j'\alpha_{0}$.

We also unwind the definition of a module over a Smith ideal. 

\begin{proposition}\label{prop-module}
If $j\mathcolon I\xrightarrow{}R$ is a Smith ideal in a closed
symmetric monoidal category $\cat{C}$, an $(R,I)$-module is equivalent
to maps of right $R$-modules $f\mathcolon M_{0}\xrightarrow{}M_{1}$
and $\phi \mathcolon M_{1}\otimes_{R}I\xrightarrow{}M_{0}$ such that
the diagrams
\[
\begin{CD}
M_{0}\otimes_{R}I @>1\otimes j>> M_{0} \otimes_{R} R \\
@Vf\otimes 1 VV @VV\cong V \\
M_{1} \otimes_{R} I @>>\phi > M_{0}
\end{CD}
\]
and 
\[
\begin{CD}
M_{1} \otimes_{R} I @>1\otimes j>> M_{1} \otimes_{R} R \\
@V\phi VV @VV\cong V \\
M_{0} @>>f> M_{1}
\end{CD}
\]
commute.  
\end{proposition}

The result for left $R$-modules is similar; this time we have maps
$f\mathcolon M_{0}\xrightarrow{}M_{1}$ and $\phi \mathcolon
I\otimes_{R}M_{1}\xrightarrow{} M_{0}$ of left $R$-modules making
analogous diagrams commute.

\begin{proof}
Let $f\mathcolon M_{0}\xrightarrow{}M_{1}$ denote an object in $\Arr
\cat{C}$.  A right $j$-module structure on $f$ is an action map $\mu
\mathcolon f\boxprod j\xrightarrow{}f$ making associativity and unit
diagrams commute.  Since we have
\[
f\boxprod j\mathcolon (M_{0}\otimes R) \amalg_{M_{0}\otimes I}
(M_{1}\otimes I) \xrightarrow{} M_{1} \otimes R,
\]
the map $\mu$ is equivalent to right multiplications $\mu_{0},\mu_{1}$
of $R$ on $M_{0}$ and $M_{1}$ that are preserved by $f$ and a map
$\phi_{0}\mathcolon M_{1}\otimes I\xrightarrow{}M_{0}$ such that
$f\phi_{0}=\mu_{1} (1\otimes j)$ and $\phi_{0} (f\otimes 1)=\mu_{0}
(1\otimes j)$.  Careful consideration of the associativity diagram
shows that $\mu$ is associative when $\mu_{0}$ and $\mu_{1}$ are
associative, when $\phi_{0}$ descends to $\phi \mathcolon
M_{1}\otimes_{R}I\xrightarrow{}M_{0}$, and when $\phi$ is a right
$R$-module map.  Of course, the unit diagram commutes exactly when the
action of $R$ on $M_{0}$ and $M_{1}$ is unital.
\end{proof}

A map of $(R,I)$-modules $M\xrightarrow{}N$ is of course a pair of
right $R$-module maps $M_{0}\xrightarrow{}N_{0}$ and
$M_{1}\xrightarrow{}N_{1}$ making the evident diagrams involving $f$
and $\phi$ commute.

Whenever we have monoids and modules, we also have extension and
restriction of scalars.  It is useful to unwind these definitions as
well. For the tensor product monoidal structure, this is simple.
Suppose $\alpha\mathcolon p\xrightarrow{}p'$ is a map of monoid
homomorphisms, where $p\mathcolon R_{0}\xrightarrow{}R_{1}$ and
$p'\mathcolon R_{0}'\xrightarrow{}R_{1}'$.  If $f\mathcolon
M_{0}\xrightarrow{}M_{1}$ is a $p$-module, then the extension of
scalars functor sends to $f$ to
\[
f\otimes_{p}p'\mathcolon M_{0}\otimes_{R_{0}}R_{0}' \xrightarrow{}
M_{1}\otimes_{R_{1}}R_{1}'.  
\]
The restriction of scalars functor sends $f$ to itself, as usual.  

For the pushout product monoidal structure, life is a bit more
complicated.    

\begin{proposition}\label{prop-extension}
Suppose $\alpha \mathcolon j\xrightarrow{}j'$ is a map of Smith ideals
in a closed symmetric monoidal category $\cat{C}$, where $j\mathcolon
I\xrightarrow{}R$ and $j'\mathcolon I'\xrightarrow{}R'$.  Let $U$
denote the restriction of scalars functor from $j'$-modules to
$j$-modules.  If $N$ is a $j'$-module, then $(UN)_{0}=N_{0}$,
$(UN)_{1}=N_{1}$, and $f_{UN}=f_{N}$, where $N_{0}$ and $N_{1}$ are
$R$-modules via restriction of scalars, and $\phi_{UN}$ is the
composite
\[
N_{1} \otimes_{R} I \xrightarrow{1\otimes \alpha } N_{1} \otimes_{R'}
I' \xrightarrow{\phi_{N}} N_{0}. 
\]
If $M$ is a $j$-module, then the extension of scalars
$M\boxprod_{j}j'$ has 
$(M\boxprod_{j}j')_{1}=M_{1}\otimes_{R}R'$, and
$(M\boxprod_{j}j')_{0}$ is the pushout in the diagram below.
\[
\begin{CD}
(M_{0}\otimes_{R}I') \amalg (M_{1}\otimes_{R} I \otimes_{R}R') @>>>
M_{0}\otimes_{R}R' \\
@VVV @VVV \\
M_{1} \otimes_{R} I' @>>\phi_{M\boxprod_{j}j'}> (M\boxprod_{j}j')_{0}
\end{CD}
\]
The map $f_{M\boxprod_{j}j'}$ is the evident one, induced by
$f_{M}\otimes 1$ and $1\otimes j'$.  
\end{proposition}

The corestriction $\Hom_{\boxprod ,j} (j',M)$ has a dual description,
using a pullback instead of a pushout, but we leave the details to the
reader.  

\begin{proof}
The only thing we need to prove is the description of
$M\boxprod_{j}j'$, which by definition is the coequalizer of the two
maps
\[
M\boxprod j\boxprod j' \rightrightarrows M\boxprod j'.  
\]
Both $\ev_{0}$ and $\ev_{1}$ preserve coequalizers, and it follows
easily that $\ev_{1} (M\boxprod_{j}j')=M_{1}\otimes_{R}R'$.  We have 
\[
\ev_{0} (M\boxprod j')= (M_{0}\otimes R')\amalg_{M_{0}\otimes I'}
(M_{1}\otimes I'),
\]
whereas $\ev_{0} (M\boxprod j\boxprod j')$ is a pushout of the three
objects 
\[
M_{0}\otimes R\otimes R', M_{1}\otimes I\otimes R', \text{ and }
M_{1}\otimes R\otimes I'.
\]
Upon taking the coequalizer, the first of these objects converts
$M_{0}\otimes R'$ to $M_{0}\otimes_{R}R'$, the second appears in our
description (suitably tensored over $R$), and the third converts
$M_{1}\otimes I'$ to $M_{1}\otimes_{R}I'$.
\end{proof}

We now return to the cokernel functor, which we recall from
Theorem~\ref{thm-cokernel-monoidal} is a symmetric monoidal functor
from the pushout product monoidal structure to the tensor product
monoidal structure.  It therefore preserves monoids and modules.  More
precisely, we have the following theorem.

\begin{theorem}\label{thm-Smith-cokernel}
Suppose $\cat{C}$ is a pointed closed symmetric monoidal category.
The cokernel induces a functor from Smith ideals in $\cat{C}$ to
monoid homomorphisms in $\cat{C}$ whose right adjoint is the kernel.
That is, if $j\mathcolon I\xrightarrow{}R$ is a Smith ideal, then the
cokernel $R\xrightarrow{}R/I$ is canonically a monoid
homomorphism. Furthermore, the cokernel also induces a functor from
modules over the Smith ideal $j$ to modules over the monoid
homomorphism $\coker j$ whose right adjoint is the kernel.
\end{theorem}

\begin{proof}
This is an immediate corollary of
Theorem~\ref{thm-cokernel-monoidal}.  Since the cokernel is (strongly)
symmetric monoidal, its right adjoint, the kernel, is lax symmetric
monoidal.  That is, there is a natural map 
\[
\ker f \boxprod  \ker g \xrightarrow{} \ker (f\otimes g)
\]
adjoint to the map 
\[
\coker (\ker f \boxprod \ker g) \cong \coker \ker f \otimes \coker
\ker g \xrightarrow{} f \otimes g,
\]
making all the usual coherence diagrams commute.  This makes the
kernel functor pass to a functor of monoids, where it is right adjoint
to the cokernel as a functor of monoids.  It also means the kernel
defines a functor from $\coker j$-modules to $j$-modules that is right
adjoint to the cokernel as well.  This too is standard, but we remind
the reader that if $f$ is a $\coker j$-module, then $\ker f$ is a
$j$-module via the multiplication 
\[
j\boxprod \ker f \xrightarrow{} \ker \coker j \boxprod \ker f
\xrightarrow{}\ker (\coker j\otimes f) \xrightarrow{} \ker f.  
\]
This completes the proof.  
\end{proof}

\section{The injective model structure on the arrow
category}\label{sec-model-injective}

In this section, we suppose that $\cat{C}$ is a closed symmetric
monoidal model category in the sense
of~\cite[Definition~4.2.6]{hovey-model}.  In this case, we would like
$\Arr \cat{C}$ to be a closed symmetric monoidal model category as
well, but we need two different model structures.  The model structure
compatible with the pushout product is the \textbf{projective model
structure}, where a map in $\Arr \cat{C}$ is a weak equivalence or
fibration if and only if its components are so in $\cat{C}$.  The
model structure compatible with the tensor product monoidal structure
is the \textbf{injective model structure}, where a map in $\Arr
\cat{C}$ is a weak equivalence or cofibration if and only if its
components are so in $\Arr \cat{C}$.  In this section, we construct
the injective model structure and establish the basic properties of
it that we need.  The main goal is to prove there is a good theory of
monoids and modules over them. 

\begin{theorem}\label{thm-injective-model}
Suppose $\cat{C}$ is a model category.  Then
there is a model structure on $\Arr \cat{C}$, called the
\textbf{injective model structure}, with the following properties\uc 
\begin{enumerate}
\item A map $\alpha$ in $\Arr \cat{C}$ is a weak equivalence \ulp
resp. cofibration\urp if and only if $\ev_{0}\alpha$ and
$\ev_{1}\alpha$ are weak equivalences \ulp resp. cofibrations\urp in
$\cat{C}$\usc
\item A map $\alpha\mathcolon f\xrightarrow{}g$ is a \ulp trivial\urp
fibration if and only if the maps $\ev_{1}\alpha $ and
\[
\ev_{0}f \xrightarrow{(f,\ev_{0}\alpha )} \ev_{1}f \times_{\ev_{1}g}
\ev_{0}g
\]
are \ulp trivial\urp fibrations in $\cat{C}$.  In particular, if
$\alpha$ is a fibration in $\Arr \cat{C}$, then both $\ev_{0}\alpha$
and $\ev_{1}\alpha$ are fibrations in $\cat{C}$.  
\item The functors $L_{0},L_{1}\mathcolon \cat{C}\xrightarrow{}\Arr
\cat{C}$ are left Quillen functors, as are $\ev_{0}$ and $\ev_{1}$.
\item If $\cat{C}$ is a symmetric monoidal model category, the tensor
product monoidal structure and the injective model structure on $\Arr
\cat{C}$ make $\Arr \cat{C}$ into a symmetric monoidal model category.
\end{enumerate}
\end{theorem}

The injective model structure on $\Arr \cat{C}$ can now be
used as input into Theorem~\ref{thm-injective-model} to obtain a
doubly injective model structure on $\Arr \Arr \cat{C}$, if desired.  

\begin{proof}
We think of the category $\cat{J}$ with two objects and one
non-identity map as an inverse category in the sense
of~\cite[Definition~5.1.1]{hovey-model}.
Then~\cite[Theorem~5.1.3]{hovey-model} implies parts~(1) and~(2),
since $\Arr \cat{C}$ is the category of $\cat{J}$-diagrams in
$\cat{C}$.  We also need the dual of~\cite[Remark~5.1.7]{hovey-model}
to see that if $\alpha$ is a fibration, then $\ev_{0}\alpha $ is a
fibration in $\cat{C}$.  Part~(3) follows, since the functors
$\ev_{i}$ preserve weak equivalences, fibrations, and cofibrations.  

For part~(4), suppose that $\alpha \mathcolon f\xrightarrow{}g$ and
$\beta \mathcolon f'\xrightarrow{}g'$ are cofibrations in the
injective model structure.  We must show that the map 
\[
(f\otimes g') \amalg_{f\otimes f'} (f'\otimes g) \xrightarrow{}
g\otimes g'
\]
is a cofibration, which is trivial if either $\alpha$ or $\beta$ is.
But the components of this map are precisely 
\[
\ev_{0}\alpha \boxprod \ev_{0}\beta \text{ and } \ev_{1}\alpha
\boxprod \ev_{1}\beta
\]
and these are guaranteed to be cofibrations, trivial if either
$\alpha$ or $\beta$ is so, by the fact that $\cat{C}$ is a monoidal
model category.  We recall that there is also a unit condition; we
need to know that the map $Q1_{S}\otimes \alpha \xrightarrow{}\alpha$
is a weak equivalence for all cofibrant $\alpha$, where $Q1_{S}$ is a
cofibrant replacement of the unit $1_{S}$ of the tensor product
monoidal structure.  In fact, if $QS$ is a cofibrant replacement of
$S$ in $\cat{C}$, then $1_{QS}$ is a cofibrant replacement of $1_{S}$
in the injective model structure.  Hence the unit axiom for $\Arr
\cat{C}$ follows from the unit axiom for $\cat{C}$.  
\end{proof}

To have a really good theory of monoids and modules over them in a
symmetric monoidal model category, though, we need to know a bit more
about the model structure.  The monoid
axiom~\cite[Definition~2.2]{schwede-shipley-monoids} guarantees that
there is an induced model structure on monoids and on modules over
them.

\begin{proposition}\label{prop-injective-model-monoid}
Let $\cat{C}$ be a model category, and give $\Arr
\cat{C}$ the injective model structure.
\begin{enumerate}
\item If $\cat{C}$ is cofibrantly generated, so is $\Arr \cat{C}$.  
\item If $\cat{C}$ is a symmetric monoidal model category and
satisfies the monoid axiom
of~\cite[Definition~2.2]{schwede-shipley-monoids}, so does $\Arr
\cat{C}$.
\end{enumerate}
\end{proposition}

\begin{proof}
Suppose $\cat{C}$ is cofibrantly generated, with generating
cofibrations $I$ and generating trivial cofibrations $J$. Let
$I'\subseteq \mor \Arr \cat{C}$ consist of the maps $L_{1}i$ for $i\in
I$ and
\[
\alpha_{i}\mathcolon i\xrightarrow{} U_{1}\ev_{1}i
\]  
for $i\in I$.  If $i\mathcolon A\xrightarrow{}B$, then $\alpha_{i}$ is
the map
\[
\begin{CD}
A @>i>> B \\
@ViVV @| \\
B @= B.
\end{CD}
\]
in $\Arr \cat{C}$.  Define $J'$ similarly using $J$.  Then one can
check that $\beta$ has the right lifting property with respect to
$L_{1}i$ if and only if $\ev_{1}\beta$ has the right lifting property
with respect to $i$, and $\beta \mathcolon f\xrightarrow{}g$ has the
right lifting property with respect to $\alpha_{i}$ if and only if
\[
\ev_{0}f \xrightarrow{} \ev_{1}f \times_{\ev_{1}g} \ev_{0}g
\]
has the right lifting property with respect to $i$.  Therefore the
maps $I'$ and $J'$ will serve as generating cofibrations and
generating trivial cofibrations for the injective model structure.
One must also check that the domains and codomains of $I'$ and $J'$
are small in $\Arr \cat{C}$, but this follows from smallness in
$\cat{C}$.  

The monoid axiom says that transfinite compositions of pushouts of
maps of the form $f\otimes A$, where $f$ is a trivial cofibration in
$\cat{C}$ and $A$ is an object of $\cat{C}$, are weak equivalences.
So in $\Arr \cat{C}$ we must check that transfinite compositions
$\beta$ of pushouts of maps of the form $\alpha \otimes g$, where
$\alpha$ is a trivial cofibration and $g$ is an object of $\Arr
\cat{C}$, are weak equivalences.  Since $\ev_{i}$ commutes with
transfinite compositions and pushouts for $i=0,1$, $\ev_{i}\beta$ is a
transfinite composition of pushouts of maps of the form
\[
\ev_{i} (\alpha \otimes g) =\ev_{i}\alpha  \otimes \ev_{i}g.
\]
Since $\ev_{i}\alpha$ is a trivial cofibration in $\cat{C}$,
$\ev_{i}\beta$ is a weak equivalence by the monoid axiom in
$\cat{C}$.  Therefore, $\beta$ is a weak equivalence as required.  
\end{proof}

\begin{corollary}\label{cor-injective-monoids}
Suppoce $\cat{C}$ is a cofibrantly generated symmetric monoidal model
category satisfying the monoid axiom.  
\begin{enumerate}
\item There is a model structure on the category of monoid
homomorphisms in $\cat{C}$, in which a map $\alpha \mathcolon
f\xrightarrow{}f'$ is a weak equivalence or fibration if and only if
it is so in the injective model structure on $\Arr \cat{C}$.  In
particular, $\alpha$ is a weak equivalence if and only if
$\ev_{0}\alpha$ and $\ev_{1}\alpha$ are weak equivalences in
$\cat{C}$.  
\item Given a monoid homomorphism $f\mathcolon
R_{0}\xrightarrow{}R_{1}$, there is a model structure on the category
of $f$-modules in which $\alpha$ is a weak equivalence or fibration if
and only if it is so in the injective model structure on $\Arr
\cat{C}$.  In particular, $\alpha$ is a weak equivalence if and only
if $\ev_{0}\alpha$ and $\ev_{1}\alpha$ are weak equivalences in
$\cat{C}$.
\end{enumerate}
\end{corollary}

This corollary follows immediately
from~\cite[Theorem~3.1]{schwede-shipley-monoids}.  Note that a monoid
homomorphism $f$ is fibrant if and only if it is a fibration of
fibrant objects in $\cat{C}$.  

But we would also like a weak equivalence of monoids to induce
a corresponding Quillen equivalence of the categories of modules, as
in~\cite[Theorem~3.3]{schwede-shipley-monoids}, so that the module
categories are homotopy invariant.  This requires a bit more. 

\begin{definition}\label{defn-flat}
Suppose $\cat{C}$ is a symmetric monoidal model category.  If $R$ is a
monoid in $\cat{C}$ and $M$ is a right $R$-module, we say that $M$ is
\textbf{flat} over $R$ if the functor $M\otimes_{R} (-)$ takes weak
equivalences of left $R$-modules to weak equivalences in $\cat{C}$.
Here a weak equivalence of left $R$-modules is an $R$-module map that
is a weak equivalence in $\cat{C}$.  We say that \textbf{cofibrant
modules are flat} in $\cat{C}$ if, for every monoid $R$, every
cofibrant right $R$-module $M$ is flat over $R$.
\end{definition}

Then Theorem~3.3 of~\cite{schwede-shipley-monoids} says that if
cofibrant modules are flat in $\cat{C}$, and $\cat{C}$ satisfies the
monoid axiom, then module categories are homotopy invariant.  

\begin{proposition}\label{prop-flat-injective}
Suppose $\cat{C}$ is cofibrantly generated and satisfies the monoid
axiom, and cofibrant modules are flat in $\cat{C}$.  Then cofibrant
modules are flat in the injective model structure on $\Arr \cat{C}$.
In particular, in this case a weak equivalence $\alpha \mathcolon
f\xrightarrow{}g$ of monoid homomorphisms induces a Quillen
equivalence of the corresponding model categories of $f$-modules and
$g$-modules.  
\end{proposition}

Before proving this proposition, we need a lemma. 

\begin{lemma}\label{lem-injective-modules}
Suppose $\cat{C}$ is cofibrantly generated and satisfies the monoid
axiom, and $f\mathcolon R_{0}\xrightarrow{}R_{1}$ is a monoid in the
tensor product monoidal structure.  Then $\ev_{0}\mathcolon f\Mod
\xrightarrow{}R_{0}\Mod$ and $\ev_{1}\mathcolon f\Mod
\xrightarrow{}R_{1}\Mod$ are left and right Quillen functors.
\end{lemma}

\begin{proof}
Since $\ev_{0}$ and $\ev_{1}$ are strict monoidal as functors from the tensor
product monoidal structure to $\cat{C}$, they induce functors from the
module categories as desired, as do their right adjoints $U_{0}$ and
$U_{1}$.  The left adjoint $L_{1}$ is also strict monoidal and so
induces a functor $L_{1}\mathcolon R_{1}\Mod \xrightarrow{}f\Mod$ left
adjoint to $\ev_{1}$.  However, $L_{0}$ is not monoidal, and the left
adjoint of $\ev_{0}\mathcolon f\Mod \xrightarrow{}R_{0}\Mod $ is instead
$L_{0}'$, where $L_{0}' (A)$ is the map
$A\xrightarrow{}R_{1}\otimes_{R_{0}}A$.  

Certainly $\ev_{0}$ and $\ev_{1}$ preserve weak equivalences.  A
fibration of $f$-modules in particular has $\ev_{0}f$ and $\ev_{1}f$
fibrations in $\cat{C}$, and so also in $R_{0}\Mod$ and $R_{1}\Mod$,
respectively.  Thus $\ev_{0}$ and $\ev_{1}$ are right Quillen
functors.  To see they are also left Quillen functors, take a
(trivial) fibration $h$ of $R$-modules.  Since $U_{0}A$ is just the
map $A\xrightarrow{}0$, one can check easily that $U_{0}h$ is a
(trivial) fibration of $f$-modules.  Similarly, if $h$ is a (trivial)
fibration of $R'$-modules, use the fact that $U_{1}A=1_{A}$ to see
that $U_{1}h$ is a (trivial) fibration of $f$-modules.
\end{proof}

\begin{proof}[Proof of Proposition~\ref{prop-flat-injective}]
Suppose $f$ is a monoid in the tensor product monoidal structure, so
that $f\mathcolon R_{0}\xrightarrow{}R_{1}$ is a homomorphism of monoids in
$\cat{C}$.  Let $g$ be a cofibrant $f$-module.  Then $\ev_{0}g$ is a
cofibrant $R_{0}$-module and $\ev_{1}g$ is a cofibrant $R_{1}$-module by
Lemma~\ref{lem-injective-modules}.  Now suppose $\alpha$ is a weak
equivalence of left $f$-modules, so that $\ev_{0}\alpha$ is a weak
equivalence of $R_{0}$-modules and $\ev_{1}\alpha$ is a weak equivalence
of $R_{1}$-modules.  Then 
\[
\ev_{0} (g\otimes_{f}\alpha)=\ev_{0}g\otimes_{R_{0}}\ev_{0}\alpha
\text{ and } \ev_{1} (g\otimes_{f}\alpha) =
\ev_{1}g\otimes_{R_{1}}\ev_{1}\alpha
\]
so the result follows from the fact that cofibrant modules are flat in
$\cat{C}$.  
\end{proof}

\section{The projective model structure on the arrow
category}\label{sec-model-projective}

In this section, we establish the projective model structure on $\Arr
\cat{C}$, which is compatible with the pushout product monoidal
structure.  Just as in the previous section, we show that there is a
good theory of monoids (Smith ideals) and modules over them, although
stronger assumptions on $\cat{C}$ are needed to get a really good
theory. In fact, to check that cofibrant modules are flat in the
projective model structure is sufficiently complicated that we discuss
it in an appendix. 

\begin{theorem}\label{thm-projective-model}
Suppose $\cat{C}$ is a model category.  Then
there is a model structure on $\Arr \cat{C}$, called the
\textbf{projective model structure}, with the following properties\uc 
\begin{enumerate}
\item A map $\alpha$ in $\Arr \cat{C}$ is a weak equivalence \ulp
resp. fibration\urp if and only if $\ev_{0}\alpha$ and
$\ev_{1}\alpha$ are weak equivalences \ulp resp. fibrations\urp in
$\cat{C}$\usc
\item A map $\alpha\mathcolon f\xrightarrow{}g$ is a \ulp trivial\urp
cofibration if and only if the maps $\ev_{0}\alpha $ and
\[
\ev_{1}f \amalg_{\ev_{0}f} \ev_{0}g \xrightarrow{(\ev_{1}\alpha ,g)} \ev_{1}g
\]
are \ulp trivial\urp cofibrations in $\cat{C}$.  In particular, if
$\alpha$ is a cofibration in $\Arr \cat{C}$, then both $\ev_{0}\alpha$
and $\ev_{1}\alpha$ are cofibrations in $\cat{C}$.  
\item The functors $L_{0},L_{1}\xrightarrow{}\cat{C}\xrightarrow{}\Arr
\cat{C}$ are left Quillen functors, as are $\ev_{0}$ and $\ev_{1}$.
\item If $\cat{C}$ is cofibrantly generated, so is the projective
model structure on $\Arr \cat{C}$.  
\item If $\cat{C}$ is a cofibrantly generated symmetric monoidal model
category, the pushout product monoidal structure and the projective
model structure on $\Arr \cat{C}$ make $\Arr \cat{C}$ into a symmetric
monoidal model category.
\end{enumerate}
\end{theorem}

The reader may well object that part~(5) of the theorem above should
not require the cofibrantly generated hypothesis.  This is surely
correct, but the proof involves so many pushout diagrams of pushout
diagrams that the author got too confused to finish the proof.  

The reader should note that the identity functor is a Quillen
equivalence from the projective model structure to the injective model
structure on $\Arr \cat{C}$.  

At least if $\cat{C}$ is cofibrantly generated, the projective model
structure can also be iterated, and we could have the doubly
projective model structure on $\Arr \Arr \cat{C}$, in addition to the
projective injective model structure and the injective projective
model structure.  We don't know if these are useful.  

\begin{proof}
We think of the category $\cat{J}$ with two objects and one
non-identity map as a direct category in the sense
of~\cite[Definition~5.1.1]{hovey-model}.
Then~\cite[Theorem~5.1.3]{hovey-model} implies parts~(1) and~(2),
since $\Arr \cat{C}$ is the category of $\cat{J}$-diagrams in
$\cat{C}$.  Then~\cite[Remark~5.1.7]{hovey-model} tells us that if
$\alpha$ is a cofibration in the projective model structure,
$\ev_{1}\alpha $ (and also $\ev_{1}\alpha$, of course) is a
cofibration in $\cat{C}$.  Part~(3) now follows easily, since the
functors $\ev_{i}$ preserve weak equivalences, fibrations, and
cofibrations.  

For part~(4), suppose $\cat{C}$ is cofibrantly generated,
with generating cofibrations $I$ and generating trivial cofibrations
$J$.  Then $\Arr \cat{C}$ is cofibrantly generated with generating
cofibrations $L_{0}I\cup L_{1}I$ and generating trivial cofibrations
$L_{0}J\cup L_{1}J$, by~\cite[Remark~5.1.8]{hovey-model}.  

We must now verify that $\Arr \cat{C}$ is a symmetric monoidal model
category when $\cat{C}$ is so (and is cofibrantly generated).  We will
check the unit condition below. For the remaining condition, because
$\Arr \cat{C}$ is cofibrantly generated, we just have to check that if
$\alpha$ is a generating cofibration and $\beta$ is a generating
(trivial) cofibration of $\Arr \cat{C}$, then the pushout product
$\alpha \boxprod_{2} \beta $ (not to be confused with the monoidal
structure $\boxprod$ in $\Arr \cat{C}$) is a (trivial) cofibration in
$\Arr \cat{C}$.  Here, if $\alpha \mathcolon f\xrightarrow{}g$ and
$\beta \mathcolon f'\xrightarrow{}g'$,
\[
\alpha \boxprod_{2} \beta \mathcolon (f\boxprod g') \amalg_{f\boxprod
f'} (g\boxprod f') \xrightarrow{} g\boxprod g'.
\]
Now, if $\alpha =L_{0}i\mathcolon L_{0}A\xrightarrow{}L_{0}B$, it
follows from Lemma~\ref{lem-eval-tensor} and some computation that 
\[
L_{0}i\boxprod_{2}\beta = L_{0} (i\boxprod \ev_{1}\beta).  
\]
In particular, if $i$ is a cofibration in $\cat{C}$ and $\beta$ is a
(trivial) cofibration in $\Arr \cat{C}$, then $\ev_{1}\beta$ is a
(trivial) cofibration in $\cat{C}$, so $i\boxprod \ev_{1}\beta$ is a
(trivial) cofibration in $\cat{C}$, and thus $L_{0}i\boxprod_{2}\beta$
is a (trivial) cofibration in $\Arr \cat{C}$.

This takes care of all of the $\alpha \boxprod_{2}\beta$ we need to
consider except $\alpha =L_{1}i$ and $\beta =L_{1}j$.  Since
$L_{1}$ is symmetric monoidal, we have 
\[
L_{1}i\boxprod_{2} L_{1}j = L_{1} (i\boxprod j).  
\]
Thus, if $i$ is a cofibration in $\cat{C}$ and $j$ is a (trivial)
cofibration in $\cat{C}$, then $i\boxprod j$ is a (trivial)
cofibration in $\cat{C}$, so $L_{1}i\boxprod_{2} L_{1}j$ is a
(trivial) cofibration in $\Arr \cat{C}$, as required.  

We are now left with checking the unit condition.  Recall that the
unit condition in $\cat{C}$ says that, if $QS\xrightarrow{}S$ is a
cofibrant replacement for the unit $S$, then $QS\otimes
A\xrightarrow{}A$ is a weak equivalence for all cofibrant $A$ in
$\cat{C}$.  In $\Arr \cat{C}$, $f$ is a cofibrant object if and only
if $f$ is a cofibration of cofibrant objects of $\cat{C}$.  It follows
that $L_{1} (QS)\mathcolon 0\xrightarrow{}QS$ is a cofibrant
replacement of the unit $L_{1}S$. It follows from Lemma~\ref{lem-eval-tensor}
that 
\[
L_{1} (QS) \boxprod f \xrightarrow{} f = QS\otimes f\xrightarrow{} f.
\]  
In particular, if $f$ is cofibrant, the domain and codomain of $f$ are
cofibrant in $\cat{C}$, so the unit axiom in $\cat{C}$ implies this
map is a weak equivalence in $\Arr \cat{C}$, as required.  
\end{proof}

Just as with the injective model structure, we would like to ensure
that good properties of the symmetric monoidal model category
$\cat{C}$ are inherited by the projective model structure on $\Arr
\cat{C}$.  There is no difficulty with the monoid axiom. 

\begin{proposition}\label{prop-monoid-ax}
If $\cat{C}$ is a cofibrantly generated symmetric monoidal model
category that satisfies the monoid axiom, then the projective model
structure on $\Arr \cat{C}$ also satisfies the monoid axiom.
\end{proposition}

\begin{proof}
According to~\cite[Lemma~2.3]{schwede-shipley-monoids}, it suffices to
check that transfinite compositions of pushouts of maps in $\Arr
\cat{C}$ of the form $L_{0}j\boxprod f$ and $L_{1}j\boxprod f$ are
weak equivalences, where $j$ is a trivial cofibration in $\cat{C}$. By
Lemma~\ref{lem-eval}, we have $L_{0}j\boxprod f=L_{0} (j\otimes
\ev_{1}f)$.  Thus 
\[
\ev_{0} (L_{0}j\boxprod f) = \ev_{1} (L_{0}j\boxprod f) = j\otimes
\ev_{1}f, 
\]
and this is a trivial cofibration in $\cat{C}$ tensored with an object
of $\cat{C}$.  Similarly, Lemma~\ref{lem-eval} tells us that 
\[
\ev_{0} (L_{1}j\boxprod f) = j\otimes \ev_{0}f \text{ and } \ev_{1}
(L_{1}j\boxprod f) = j\otimes \ev_{1}f.  
\]
Therefore, these maps are also trivial cofibrations in $\cat{C}$
tensored with objects of $\cat{C}$.  Since $\ev_{0}$ and $\ev_{1}$ are
left adjoints, they commute with transfinite compositions and
pushouts.  Thus, if we apply $\ev_{0}$ or $\ev_{1}$ to a transfinite
composition of pushouts of maps of the form $L_{0}j\boxprod f$ and
$L_{1}j\boxprod f$, we will get a transfinite composition of pushouts
of maps of the form $j\otimes X$, which are weak equivalences in
$\cat{C}$ as required, because $\cat{C}$ satisfies the monoid axiom.  
\end{proof}

\begin{corollary}\label{cor-projective-monoids}
Suppose $\cat{C}$ is a cofibrantly generated symmetric monoidal model
category satisfying the monoid axiom. 
\begin{enumerate}
\item There is a model structure on the category of Smith ideals in
$\cat{C}$, in which a map $\alpha \mathcolon 
j\xrightarrow{}j'$ is a weak equivalence or fibration if and only if
$\ev_{0}\alpha$ and $\ev_{1}\alpha$ are weak equivalences or
fibrations in $\cat{C}$.  
\item Given a Smith ideal $j$, there is a model structure on the
category of $j$-modules in which $\alpha$ is a weak equivalence or
fibration if and only if $\ev_{0}\alpha$ and $\ev_{1}\alpha$ are weak
equivalences or fibrations in $\cat{C}$.
\end{enumerate}
\end{corollary}

This corollary follows
from~\cite[Theorem~3.1]{schwede-shipley-monoids}.  

It is also useful to point out the following fact.  

\begin{corollary}\label{cor-projective-cofibrant}
Suppose $\cat{C}$ is a cofibrantly generated symmetric monoidal model
category satisfying the monoid axiom.
\begin{enumerate}
\item The functor $\ev_{1}$ is a left Quillen functor from Smith
ideals in $\cat{C}$ to monoids in $\cat{C}$.  In particular, the
codomain of a cofibrant Smith ideal is a cofibrant monoid.  
\item If the unit $S$ is cofibrant in $\cat{C}$, then a cofibrant
Smith ideal is, in particular, a cofibration of cofibrant objects in
$\cat{C}$.  
\end{enumerate}
\end{corollary}

\begin{proof}
The functor $\ev_{1}$ is strict monoidal with respect to the pushout
product monoidal structure, so does induce a functor from Smith ideals
to monoids.  Its right adjoint $U_{1}$ takes $R$ to $1_{R}$, and this
obviously perserves weak equivalences and fibrations.  Thus $\ev_{1}$
is a left Quillen functor.  

If the unit $S$ is cofibrant in $\cat{C}$, then the unit
$0\xrightarrow{}S$ of the pushout product monoidal structure is
cofibrant in $\Arr \cat{C}$, and so a cofibrant Smith ideal is in
particular cofibrant in the projective model structure on $\Arr
\cat{C}$ by~\cite[Theorem~3.1]{schwede-shipley-monoids}. 
\end{proof}

We would also like to know that if cofibrant modules are flat in
$\cat{C}$ (see Definition~\ref{defn-flat}), then cofibrant modules are
flat in the projective model structure on $\Arr \cat{C}$.  This seems
to be more complicated than the corresponding question for the
injective module structure, so we postpone this discussion to the
appendix, where we will prove the following theorem.  The definition
of a pure class of morphisms is Definition~\ref{defn-pure}.  

\begin{theorem}\label{thm-pure-Smith}
Suppose $\cat{C}$ is a cofibrantly generated symmetric monoidal model
category satisfying the monoid axiom, and $\cat{P}$ is a pure class of
morphisms of $\cat{C}$ containing the maps $i\otimes X$ for all
generating cofibrations $i$ of $\cat{C}$ and all $X\in \cat{C}$.
Assume that the domains and codomains of the generating cofibrations
of $\cat{C}$ are flat.  Then cofibrant modules are flat in the
projective model structure on $\Arr \cat{C}$.  In this case, a weak
equivalence of Smith ideals induces a Quillen equivalence of the
corresponding module categories.  
\end{theorem}

\section{The cokernel functor in the stable case}\label{sec-stable}

In this section, we prove that the cokernel functor is a left Quillen
functor from the projective model structure on $\Arr \cat{C}$ to the
injective model structure.  Even better, when $\cat{C}$ is stable, in
the sense that $\sho \cat{C}$ is triangulated, the cokernel functor is
a Quillen equivalence.  Thus a Smith ideal is the same thing as a
monoid homomorphism, up to homotopy. 

\begin{proposition}\label{prop-cokernel-Quillen}
Suppose $\cat{C}$ is a pointed model category.  The
cokernel is a left Quillen functor from the projective model structure
on $\Arr \cat{C}$ to the injective model structure.
\end{proposition}

\begin{proof}
Consider the category $\cat{J}$ with three objects $-1,0,1$ and two
non-identity morphisms $0\xrightarrow{}1$ and $0\xrightarrow{}-1$.  We
consider the first map as raising degree, and the second as lowering
degree.  This makes our category a Reedy
category~\cite[Definition~5.2.1]{hovey-model}.  There is therefore a
model structure on $\cat{J}$-diagrams on
$\cat{C}$~\cite[Theorem~5.2.5]{hovey-model}.  This is the model
structure used in the proof of the cube
lemma~\cite[Lemma~5.2.6]{hovey-model}, where it is proved that the
colimit, which is just the pushout, is a left Quillen functor from
$\cat{C}^{\cat{J}}$ to $\cat{C}$.  

Now suppose $\alpha \mathcolon f\xrightarrow{}g$ is a (trivial)
cofibration in the projective model structure.  Write $f\mathcolon
X_{0}\xrightarrow{}X_{1}$ and $g\mathcolon Y_{0}\xrightarrow{}Y_{1}$.
We then have the associated objects of $\cat{C}^{\cat{J}}$, namely 
\[
\begin{CD}
X_{0} @>f>> X_{1} \\
@VVV \\
0
\end{CD}
\]
and 
\[
\begin{CD}
Y_{0} @>g>> Y_{1} \\
@VVV \\
0.
\end{CD}
\]
The map $\alpha $ induces a map between these diagrams, and this map
is a (trivial) cofibration in $\cat{C}^{\cat{J}}$.  Indeed, because
the map $0\xrightarrow{}-1$ in $\cat{J}$ lowers degree, we just need 
\[
1_{0}\mathcolon 0\xrightarrow{}0, \alpha_{0}\mathcolon
X_{0}\xrightarrow{} Y_{0} \text{ and } (\alpha_{1},g)\mathcolon
X_{1}\amalg_{X_{0}} Y_{0} \xrightarrow{} Y_{1}
\]
to be (trivial) cofibrations, as of course they are.  We conclude that
the induced map 
\[
\coker f\xrightarrow{}\coker g
\]
of the colimits is a (trivial) cofibration, which is just what we need
to make $\coker$ a left Quillen functor from the projective to the
injective model structure.  
\end{proof}

\begin{corollary}\label{cor-cokernel-Quillen}
Suppose $\cat{C}$ is a cofibrantly generated pointed symmetric
monoidal model category.  
\begin{enumerate}
\item The cokernel is a left Quillen functor from the model category of
Smith ideals in $\cat{C}$ to the model cateory of monoid homomorphisms
in $\cat{C}$. 
\item If $j$ is a Smith ideal in $\cat{C}$, the cokernel induces a
left Quillen functor from the model category of $j$-modules to the
model category of $\coker j$-modules.  
\end{enumerate}
\end{corollary}

\begin{proof}
This follows from the general theory.  The right adjoint $\ker$
preserves (trivial) fibrations in $\Arr \cat{C}$, and since (trivial)
fibrations of monoids or modules are just maps of monoids or modules
that are (trivial) fibrations in $\Arr \cat{C}$, $\ker$ will preserve
(trivial) fibrations of monoids and modules.
\end{proof}

We now suppose that $\cat{C}$ is \textbf{stable}, so that $\sho
\cat{C}$ is a triangulated category~\cite[Chapter~7]{hovey-model}.
Since a map of exact triangles in a triangulated category that is an
isomorphism on two of the three spots is an isomorphism on the third
spot as well, we should expect the cokernel to be a Quillen
equivalence in this case.  This is in fact true.

\begin{theorem}\label{thm-cokernel-stable}
Suppose $\cat{C}$ is a stable model category.  Then the cokernel is a
Quillen equivalence from the projective model structure on $\Arr
\cat{C}$ to the injective model structure.
\end{theorem}

\begin{proof}
Suppose $f$ is cofibrant in the projective model structure on $\Arr
\cat{C}$ and $p$ is fibrant in the injective model structure.  This
means that $f\mathcolon A\xrightarrow{}B$ is a cofibration of
cofibrant objects and $p\mathcolon X\xrightarrow{}Y$ is a fibration of
fibrant objects.  Let $\alpha \mathcolon \coker f \xrightarrow{}p$
be a map, which we write 
\[
\begin{CD}
B @>g>> \coker f \\
@V\alpha_{0}VV @VV\alpha_{1}V \\
X @>>p> Y,
\end{CD}
\]
with corresponding map $\beta \mathcolon f\xrightarrow{}\ker p$,
written 
\[
\begin{CD}
A @>f>> B \\
@V\beta_{0}VV @VV\alpha_{0}V \\
\ker p @>>q> X.
\end{CD}
\]
We must show that a map $\alpha $ is a weak equivalence if and only if
$\beta$ is.  In the homotopy category $\sho \cat{C}$, the map $f$
gives rise to a cofiber sequence
\[
A \xrightarrow{f} B \xrightarrow{}\coker f \xrightarrow{} \Sigma A
\]
and the map $g$ gives rise to a fiber sequence
\[
\Omega Y \xrightarrow{} \ker p \xrightarrow{q} X \xrightarrow{p} Y.
\]
Since $\cat{C}$ is stable, every fiber sequence is (canonically
isomorphic to) a cofiber sequence, and so 
\[
\ker p \xrightarrow{q} X \xrightarrow{p} Y \xrightarrow{} \Sigma (\ker p)
\]
is a cofiber sequence,  Then $\alpha$ and $\beta$ together define a
map between two exact triangles in a triangulated category; if
$\alpha$ or $\beta$ is a weak equivalence, then this map is an
isomorphism at two out of the three spots, so also an isomorphism at
the third spot.  Thus $\alpha$ is a weak equivalence if and only if
$\beta$ is so.  
\end{proof}

\begin{corollary}\label{cor-cokernel-Quillen-stable}
Suppose $\cat{C}$ is a cofibrantly generated stable symmetric
monoidal model category in which the unit $S$ is cofibrant.  
\begin{enumerate}
\item The cokernel is a Quillen equivalence from the model category of
Smith ideals in $\cat{C}$ to the model category of monoid homomorphisms
in $\cat{C}$.
\item If $j$ is a cofibrant Smith ideal in $\cat{C}$, the cokernel
induces a Quillen equivalence from the model category of $j$-modules
to the model category of $\coker j$-modules.
\end{enumerate}
\end{corollary}

\begin{proof}
The kernel, as a functor from monoid homomorphisms to Smith ideals,
reflects weak equivalences between fibrant objects.  Indeed, since
fibrations and weak equivalences are created in the underlying model
structures on $\Arr \cat{C}$, this follows from fact that the kernel,
as the right half of a Quillen equivalence on $\Arr \cat{C}$, reflects
weak equivalences between fibrant
objects~\cite[Corollary~1.3.16]{hovey-model}.  

Now suppose $j$ is a cofibrant Smith ideal.  Since $S$ is cofibrant,
the unit $L_{1}S$ of the pushout product monoidal structure is
cofibrant in the projective model structure.  Hence a cofibrant monoid
is in particular cofibrant in $\Arr \cat{C}$
by~\cite[Theorem~3.1]{schwede-shipley-monoids}.  Thus $j$ is cofibrant
in $\Arr \cat{C}$.  This means that $j\xrightarrow{}\ker T (\coker j)$
is a weak equivalence, where $T$ is a fibrant replacement functor in
the injective model structure on $\Arr \cat{C}$, again
using~\cite[Corollary~1.3.16]{hovey-model}.  If $T'$ denotes a fibrant
replacement functor in the category of monoid homomorphisms, there is
a weak equivalence $T (\coker j)\xrightarrow{}T' (\coker j)$ in the
injective model structure.  These are both fibrant objects, and so the
induced map $\ker T (\coker j)\xrightarrow{}\ker T' (\coker j)$ is a
weak equivalence.  Thus the map $j\xrightarrow{}\ker T' (\coker j)$ is
a weak equivalence, and so the cokernel is a Quillen equivalence from
Smith ideals to monoid homomorphisms
by~\cite[Corollary~1.3.16]{hovey-model}.  

Now again assume $j$ is a cofibrant Smith ideal.  The kernel again
reflects weak equivalences between fibrant $\coker j$-modules.  We
claim that a cofibrant $j$-module is also cofibrant in $\cat{C}$, so
that we can repeat the above argument to complete the proof that the
cokernel is a Quillen equivalence from $j$-modules to $Fj$-modules.
Indeed, a cofibrant object in any cofibrantly generated model category
is a transfinite extension of the cokernels of the generating
cofibrations.  In our case, these cokernels are of the form $j\boxprod
q$, where $q$ is a cokernel of a generating cofibration in $\Arr
\cat{C}$ and hence cofibrant in $\Arr \cat{C}$.  Since $j$ is also
cofibrant in $\Arr \cat{C}$, we conclude that cofibrant $j$-modules
are cofibrant in $\Arr \cat{C}$.  
\end{proof}

\appendix
\section{Pure classes of morphisms}\label{sec-appendix}

In this appendix, we prove Theorem~\ref{thm-pure-Smith} about when
cofibrant modules are flat in the projective model structure on $\Arr
\cat{C}$.  

Consider the more general question of when cofibrant $R$-modules are
flat, for $R$ a monoid in a symmetric monoidal model category
$\cat{C}$.  The logical way to approach this is to build up from the
generating cofibrations in $R\Mod$ to all cofibrant $R$-modules.  So
we should have some result that asserts that if the domains and
codomains of the generating cofibrations in $\cat{C}$ are flat, then
all cofibrant $R$-modules are flat, for any $R$.  A cofibrant
$R$-module is a retract of a transfinite composition of pushouts of
maps $i\otimes 1\mathcolon A\otimes R\xrightarrow{}B\otimes R$, where
$i\mathcolon A\xrightarrow{}B$ is a generating cofibration of
$\cat{C}$.  In general, pushouts of maps in model categories only
behave well when the maps are cofibrations, and we cannot expect
$i\otimes 1$ to be a cofibration.  But it might be something a little
weaker.

\begin{definition}\label{defn-pure}
Define a class of morphisms $\cat{P}$ in a model category $\cat{C}$ to be a
\textbf{pure class} if the following properties hold.
\begin{enumerate}
\item A pushout of a map in $\cat{P}$ is in $\cat{P}$.  
\item If we have a map of diagrams 
\[
\begin{CD}
B @<<< A @>f>> C \\
@VVV @VVV @VVV \\
B' @<<< A' @>>f'> C'
\end{CD}
\]
in which the vertical maps are weak equivalences, and $f$ and $f'$ are
in $\cat{P}$, then the induced map of pushouts 
\[
B\amalg_{A} C \xrightarrow{} B' \amalg_{A'} C'
\]
is a weak equivalence.  
\item If $\lambda$ is an ordinal, $X,Y\mathcolon \lambda
\xrightarrow{}\cat{C}$ are colimit-preserving functors such that each
map $X_{\alpha}\xrightarrow{}X_{\alpha +1}$ and
$Y_{\alpha}\xrightarrow{}Y_{\alpha +1}$ is in $\cat{P}$, and
$f\mathcolon X\xrightarrow{}Y$ is a natural transformation such that
$f_{\alpha}$ is a weak equivalence for all $\alpha <\lambda$, then the
induced map 
\[
\colim_{\alpha <\lambda} X_{\alpha} \xrightarrow{} \colim_{\alpha
<\lambda} Y_{\alpha}
\]
is a weak equivalence.  
\end{enumerate}
\end{definition}

The basic example of a pure class of morphisms is the class of
cofibrations in a left proper model category,
by~\cite[Proposition~13.5.4]{hirschhorn}
and~\cite[Proposition~17.9.3]{hirschhorn}.

The main advantage of a pure class of morphisms is the following theorem.  

\begin{theorem}\label{thm-pure}
Suppose $\cat{C}$ is a cofibrantly generated symmetric monoidal model
category satisfying the monoid axiom, and $\cat{P}$ is a pure class of
morphisms of $\cat{C}$ containing the maps $i\otimes X$ for all
generating cofibrations $i$ of $\cat{C}$ and all $X\in \cat{C}$.
Assume that the domains and codomains of the generating cofibrations
of $\cat{C}$ are flat.  Then cofibrant $R$-modules are flat, for any
monoid $R$ in $\cat{C}$.
\end{theorem}

\begin{proof}
Since retracts of flat $R$-modules are flat, we can assume our
cofibrant $R$-module $X=\colim_{\alpha <\lambda } X_{\alpha }$ for
some ordinal $\lambda$ and some colimit-preserving functor, where each
map $X_{\alpha}\xrightarrow{}X_{\alpha +1}$ is a pushout of a
generating cofibration of $R\Mod$ and $X_{0}=0$.  We prove that
$X_{\alpha}$ is flat for all $\alpha \leq \lambda$ by transfinite
induction, the base case $X_{0}=0$ being obvious.  Let $f\mathcolon
M\xrightarrow{}N$ be a weak equivalence of left $R$-modules.

For the successor ordinal case, we have a pushout 
\[
\begin{CD}
A \otimes R @>i\otimes 1>> B\otimes R \\
@VVV @VVV \\
X_{\alpha} @>>> X_{\alpha +1}
\end{CD}
\]
where $i$ is a generating cofibration of $\cat{C}$.  We then get a map
of pushout squares from
\[
\begin{CD}
A\otimes M @>>> B \otimes M \\
@VVV @VVV \\
X_{\alpha}\otimes_{R} M @>>> X_{\alpha +1}\otimes_{R} M
\end{CD}
\]
to 
\[
\begin{CD}
A\otimes N @>>> B \otimes N \\
@VVV @VVV \\
X_{\alpha}\otimes_{R} N @>>> X_{\alpha +1}\otimes_{R} N.
\end{CD}
\]
Since $A$ and $B$ are flat in $\cat{C}$, and $X_{\alpha}$ is a a flat
$R$-module by the induction hypothesis, this map is a weak equivalence
at all of the corners except the bottom-right corner.  Since the maps
$i\otimes M$ and $i\otimes N$ are in the pure class $\cat{P}$, we
conclude that $X_{\alpha +1}\otimes_{R}f$ is a weak equivalence as
well, concluding the successor ordinal case.  However, we also note
that 
\[
X_{\alpha}\otimes_{R} M \xrightarrow{} X_{\alpha +1} \otimes_{R} M,
\]
and the analogous map for $N$, is in $\cat{P}$.

For the limit ordinal case, suppose $X_{\alpha}$ is flat for all
$\alpha <\beta$.  We have a map of diagrams 
\[
X_{\alpha}\otimes_{R}M\xrightarrow{} X_{\alpha}\otimes_{R}N
\]
that is a weak equivalence for all $\alpha <\beta$, and furthermore
each map $X_{\alpha}\otimes_{R}M \xrightarrow{} X_{\alpha
+1}\otimes_{R}M$ is in $\cat{P}$, as is the corresponding map for
$N$.  Thus 
\[
X_{\beta}\otimes_{R} M \xrightarrow{} X_{\beta} \otimes_{R} N
\]
is a weak equivalence, since $\cat{P}$ is a pure class, completing the
induction.  
\end{proof}

Now we return to $\Arr \cat{}$.  For this we need to transfer our pure
class from $\cat{C}$ to $\Arr \cat{C}$.  Fortunately this is
straightforward, and we get the following theorem.

\begin{theorem}\label{thm-pure-Smith-2}
Suppose $\cat{C}$ is a cofibrantly generated symmetric monoidal model
category satisfying the monoid axiom, and $\cat{P}$ is a pure class of
morphisms of $\cat{C}$ containing the maps $i\otimes X$ for all
generating cofibrations $i$ of $\cat{C}$ and all $X\in \cat{C}$.
Assume that the domains and codomains of the generating cofibrations
of $\cat{C}$ are flat.  Then cofibrant modules are flat in the
projective model structure on $\Arr \cat{C}$.
\end{theorem}

\begin{proof}
The plan is to apply Theorem~\ref{thm-pure} to $\Arr
\cat{C}$.  The generating cofibrations of $\Arr \cat{C}$ consist of
the maps $L_{0}i$ and $L_{1}i$ for generating cofibrations
$i\mathcolon A\xrightarrow{}B$ of $\cat{C}$.  So our first job is to
check that $L_{0}$ and $L_{1}$ preserve flat objects.  Suppose $\alpha
\mathcolon f\xrightarrow{}g$ is a weak equivalence in $\Arr \cat{C}$.
Applying Lemma~\ref{lem-eval}, we find that 
\[
\ev_{i} (L_{0} (X) \boxprod \alpha ) = X\otimes \ev_{1}\alpha 
\]
and 
\[
\ev_{i} (L_{1} (X)\boxprod \alpha) = X\otimes \ev_{i}\alpha 
\]
for $i=0,1$.  If $X$ is flat, $X\otimes \ev_{i}\alpha$ is a weak
equivalence, and so both $L_{0}$ and $L_{1}$ preserve flat objects.
We conclude that the domains and codomains of the generating
cofibrations of $\Arr \cat{C}$ are flat.  

Let $\cat{Q}$ denote the class of morphisms $f$ in $\Arr \cat{C}$ so
that $\ev_{0}f$ and $\ev_{1}f$ are in $\cat{P}$.  Since colimits and
weak equivalences in $\Arr \cat{C}$ are detected in $\cat{C}$, it is
easy to see that $\cat{Q}$ is a pure class of morphisms in $\Arr
\cat{C}$.  It remains to show that both $L_{0}i\boxprod f$ and
$L_{1}i\boxprod f$ are in $\cat{Q}$ for $i$ a generating cofibration
of $\cat{C}$ and $f$ any object of $\Arr \cat{C}$.  But, again using
Lemma~\ref{lem-eval}, we find that 
\[
\ev_{i} (L_{0}i\boxprod f) = i\otimes \ev_{1}f
\]
and 
\[
\ev_{i} (L_{1}i\boxprod f) = i\otimes \ev_{i}f
\]
for $i=0,1$.  Since both of these are in $\cat{P}$, we conclude that
$L_{0}i\boxprod f$ and $L_{1}i\boxprod f$ are in $\cat{Q}$.  Hence we
can apply Theorem~\ref{thm-pure}  to $\Arr \cat{C}$ to get our
theorem.  
\end{proof}


\providecommand{\bysame}{\leavevmode\hbox to3em{\hrulefill}\thinspace}
\providecommand{\MR}{\relax\ifhmode\unskip\space\fi MR }
\providecommand{\MRhref}[2]{%
  \href{http://www.ams.org/mathscinet-getitem?mr=#1}{#2}
}
\providecommand{\href}[2]{#2}

\end{document}